\documentclass[a4paper]{article}

\usepackage{amsmath}
\usepackage{amssymb}
\usepackage[pdftex,bookmarksnumbered=true]{hyperref}
\usepackage[pdftex]{graphicx,color} 
\usepackage{tikz-cd}  % cd = 'commutative diagramm'

\newtheorem{lemma}[subsection]{Lemma}
\newtheorem{proposition}[subsection]{Proposition}
\newtheorem{theorem}[subsection]{Theorem}
\newtheorem{corollary}[subsection]{Corollary}
\newtheorem{claim}[subsection]{Claim}
\newtheorem{fact}[subsection]{Fact}
\newenvironment{remark}%
   {\refstepcounter{subsection}%
        \medbreak\noindent{\bf Remark \thesubsection\space}}%
   {\par\medbreak}%
\newenvironment{example}%
   {\refstepcounter{subsection}%
        \medbreak\noindent{\bf Example \thesubsection\space}}%
   {\par\medbreak}%
\newenvironment{proof}[1][\unskip]%
   {\medbreak\noindent{\it Proof #1:\space}}%
   {\par\noindent\vrule height 5pt width 5pt depth 0pt\smallbreak}%
\newcommand{\df}{\bf}

\let\sauvegardetiret=\-
\renewcommand{\-}[1]{\ifx#1-\penalty10000\hbox{-\relax}\penalty10000\else\sauvegardetiret#1\fi}
\newcommand{\NN}{{\rm\bf N}}
\newcommand{\cB}{{\cal B}}
\newcommand{\cC}{{\cal C}}
\newcommand{\cD}{{\cal D}}
\newcommand{\cH}{{\cal H}}

\newcommand{\cL}{{\cal L}}
\newcommand{\cO}{{\cal O}}
\newcommand{\cS}{{\cal S}}

\newcommand{\cV}{{\cal V}}
\newcommand{\cX}{{\cal X}}
\newcommand{\cY}{{\cal Y}}
\newcommand{\HH}{\cH_0}

\newcommand{\tq}{\mathop{:}}

\let \join = \vee  % join = vee
\let \meet = \wedge  % meet = wedge
\newcommand{\jjoin}{\mathop{\join\mskip-6mu\relax\join}}
\newcommand{\mmeet}{\mathop{\meet\mskip-6mu\relax\meet}}

\let \lor = \bigvee  % lor = bigvee
\let \land = \bigwedge  % land = bigwedge
\newcommand{\llor}{\mathop{\lor\mskip-11mu\relax\lor}}
\newcommand{\lland}{\mathop{\land\mskip-11mu\relax\land}}

\newcommand{\oc}{[\![}
\newcommand{\fc}{]\!]}

\newcommand{\UN}{\mathbf{1}}
\newcommand{\ZERO}{\mathbf{0}}

\let\SauveUA=\uparrow
\renewcommand{\uparrow}{{\SauveUA}}
\let\SauveDA=\downarrow
\renewcommand{\downarrow}{{\SauveDA}}

\newcommand{\CD}[1]{(CD$_{#1}$)}

\newcommand{\dd}[1]{d_{#1}}
\newcommand{\ddbar}[1]{\dd{\bar #1}}
\newcommand{\EQUATION}[5]{{\cal #1}_{#2,\bar #3}(\bar #4,#5)}
\newcommand{\EQparsol}{\EQUATION S t s}
\newcommand{\EQpar}[1]{\EQparsol {#1} q}
\newcommand{\ES}[1]{X_{#1}}
\newcommand{\ESF}[1]{\ES{\Free{#1}}}
\newcommand{\ESFbar}[1]{\ES{\Freebar{#1}}}
\newcommand{\Free}[1]{F({#1})}
\newcommand{\Freebar}[1]{\Free{\bar #1}}
\newcommand{\Llat}{\cL_{\text{lat}}}
\newcommand{\LHA}{\cL_{\text{HA}}}

\DeclareMathOperator{\EQUIV}{Equiv}
\DeclareMathOperator{\FC}{FC}

\DeclareMathOperator{\fKer}{Ker^\uparrow}

\DeclareMathOperator{\Spec}{Spec}
\DeclareMathOperator{\Th}{\text{Th}}

\title{On the model-completion of Heyting algebras}
\author{Luck Darni\`ere\footnote{D\'epartement de math\'ematiques, Facult\'e
des sciences, 2 bd Lavoisier, 49045 Angers, France.}}

\begin{document}

\maketitle

\begin{abstract}
  We axiomatize the model-completion of the theory of Heyting algebras
  by means of the ``Density'' and ``Splitting'' properties in
  \cite{darn-junk-2018}, and of a certain ``QE Property'' that we
  introduce here. In addition: we prove that this model-completion has
  a prime model, which is locally finite and which we explicitly
  construct; we show how the Open Mapping Theorem of
  \cite{gool-regg-2018} can be derived from the QE Property of
  existentially closed Heyting algebras; and we construct a certain
  ``discriminant'' for equations in Heyting algebras, similar to its
  ring theoretic counterpart. 
\end{abstract}

\section{Introduction}
\label{se:intro}

It is known since \cite{pitt-1992} that the second-order intuitionist
propositional calculus (IPC$^2$) is interpretable in the first-order
one (IPC). This proof-theoretic result, after translation in
model-theoretic terms, ensures that the theory of Heyting algebras has
a model-completion (see \cite{ghil-zawa-1997}). It has been followed
by numerous investigations on uniform interpolation, intermediate
logics and model-theory (\cite{krem-1997}, \cite{pola-1998},
\cite{ghil-1999}, \cite{bilk-2007}, \cite{darn-junk-2018},
\cite{gool-regg-2018}, among others). In particular, it is proved in
\cite{ghil-zawa-1997} that each of the height varieties of Heyting
algebras which has the amalgamation property also has a
model-completion. In spite of all this work, no intuitively meaningful
axiomatization was known by now for the model-completion of the theory
of Heyting algebras. So the algebraic nature of its models, the
existentially closed Heyting algebras, remains quite mysterious. Our
goal in this paper is to fill this lacuna.

Recently, van Gool and Reggio have given in \cite{gool-regg-2018} a
different proof of Pitts' result, by  showing that it can easily be
derived from the following ``Open Mapping Theorem''. 

\begin{theorem}[Open Mapping] \label{th:G-R}
  Every continuous p-morphism between finitely co-presented Esakia
  spaces is an open map. 
\end{theorem}

That both results are tightly connected is further illustrated by the
approach in this paper: we prove that, reciprocally, the Open Mapping
Theorem can easily be derived from Pitts' result. 

This, and our axiomatization of the model-completion of the theory of
Heyting algebras, is based on a new property of existentially closed
Heyting algebras that we introduce here, the ``QE property''. Its name
comes from the following characterisation: a Heyting algebra has the
QE property if and only if its complete theory eliminates the
quantifiers (Proposition~\ref{pr:QE}). The smart reader may observe
that, knowing by Pitts' result that the theory of existentially closed
Heyting algebras eliminates the quantifiers, it is not an achievement
to prove that all its models have the QE property!

But things are a bit more complicated: given a model $M$ of a theory
$T$ having a model-completion $T^*$ , if $M\models T^*$ then obviously the
theory of $M$ eliminates the quantifiers, however the converse is not true
in general. For example every dense Boolean algebra has the
QE-property (because the theory of dense Boolean algebras eliminates
the quantifiers), {\em but} none of them is existentially closed in
the class of all Heyting algebras. Moreover the QE-property itself is
probably not first-order in general. Our strategy in the present paper
is then: firstly to characterize {\em among} the Heyting algebras
having the QE property, those which are existentially closed; and
secondly to prove that the QE-property is first-order axiomatisable in
this class. This is done by showing first that a certain Heyting
algebra $\HH$, obtained by amalgamating appropriately a set of
representatives of all the finite Heyting algebras, is existentially
closed. By Pitts' result $\HH$ is then a model of the model-completion
of the theory of Heyting algebras, hence it has the QE property.
Moreover it is locally finite by construction, and a prime model of
this model-completion. Actually it is the only countable and
existentially closed Heyting algebra which is locally finite. 
\\

This paper is organised as follows. We recall the needed prerequisites
in Section~\ref{se:notation}. In Section~\ref{se:prime-model} we
construct the prime model $\HH$. In Section~\ref{se:hensel}, after
introducing the QE property, we use that $\HH$ is a prime model to
prove that a Heyting algebra $H$ is existentially closed if and only
if it has the QE property {\em and} the ``Splitting'' and ``Density''
properties that we introduced in \cite{darn-junk-2018}. In the next
Section~\ref{se:OMT} we show how the Open Mapping Theorem can be
derived from this. Finally we turn to explicit elimination. Indeed the
elimination of quantifiers is done by Pitts in a very effective way,
but it can be somewhat puzzling from an algebraic point of view. So in
Section~\ref{se:disc} it seemed to us that it could be worth showing
how, given a finite system of equations
\begin{displaymath}
  t(\bar p,q)=\UN\quad\mbox{and}\quad s_k(\bar p,q)\neq\UN\quad
  (\mbox{for }1\leq k\leq \kappa),
\end{displaymath}
we can use the Open Mapping Theorem to explicitly construct a term
$\Delta_t(\bar p)$ (the ``discriminant'' of $t$) and terms $\nabla_{t,s_k}(\bar
p)$ for $1\leq k\leq \kappa$ (the ``co-discriminant'' of $t$ and $s_k$), such
that for every tuple $\bar a$ in a Heyting algebra $A$, the above
system (with parameters $\bar a$ replacing $\bar p$), has a solution in
an extension of $A$ if and only if $\Delta_t(\bar a)=\UN$ and
$\nabla_{t,s_k}(\bar a)\neq\UN$ for $1\leq k\leq \kappa$. The appendix
Section~\ref{se:appendix} is devoted to the Finite Extension Property,
a strengthening of the Finite Model Property of IPC which plays a key
role in proving that the Heyting algebra constructed
Section~\ref{se:prime-model} is indeed existentially closed. 

\begin{remark}\label{re:V2}
  It is claimed in \cite{ghil-zawa-1997} that the methods of Pitts, which
  concern the variety $\cV_1$ of all Heyting algebras, apply {\it
  mutatis mutandis} to the variety $\cV_2$ of the logic of the weak
  excluded middle (axiomatized by $\lnot  x\join \lnot \lnot  x = \UN$). Consequently
  all the results presented here for $\cV_1$ remain true for $\cV_2$.
  It suffices to restrict the amalgamation in
  Section~\ref{se:prime-model} to the finite algebras in $\cV_2$, and
  to replace the ``Density'' and ``Splitting'' axioms for $\cV_1$ (D1
  and S1 in \cite{darn-junk-2018}) by the corresponding axioms for
  $\cV_2$ (D2 and S2 in \cite{darn-junk-2018}). 
\end{remark}

\section{Notation and prerequisites}
\label{se:notation}

\paragraph{Posets and lattices.}
For every element $a$ in a poset $(E,\leq)$ we let $a\uparrow=\{b\in E\tq a\leq b\}$
For every subset $S$ of $E$ we let $S\uparrow=\bigcup_{s\in S}s\uparrow$ be the set of
elements of $E$ greater than or equal to some element of $S$. An
{\df up-set} of $E$ is a set $S\subseteq E$ such that $S=S\uparrow$. The family of
all up-sets of $E$ forms a topology on $E$ that we denote by
$\cO^\uparrow(E)$. 

All our lattices will be assumed distributive and bounded. The
language of lattices is then $\Llat=\{\ZERO,\UN,\join,\meet\}$, where $\ZERO$
stands for the smallest element, $\UN$ for the greatest one, $\join$ and
$\meet$ for the join and meet operations respectively. Iterated joins
and meets will be denoted $\jjoin$ and $\mmeet$. Of course the order of the
lattice is definable by $b\leq a$ if and only if $a\join b=a$. 
The language of Heyting algebras is $\LHA=\Llat\cup\{\to\}$ where the new
binary operation symbol is interpreted as usually by
\begin{displaymath}
b\to a=\max\big\{c\tq c\meet b\leq a\big\}.
\end{displaymath}

\paragraph{Stone spaces and morphisms.} 
The prime filter spectrum, or Stone space, of a distributive bounded
lattice $A$ is denoted $\Spec^\uparrow A$. The {\df Zariski topology} on
$\Spec^\uparrow A$ is generated by the sets
\begin{displaymath}
  P^\uparrow(a)=\big\{x\in\Spec^\uparrow A\tq a\in x\big\}
\end{displaymath}
where $a$ ranges over $A$. The {\df patch topology} is generated by
the sets $P^\uparrow(a)$ and their complements in $\Spec^\uparrow A$.

Stone's duality states that $a\mapsto P^\uparrow(a)$ is an embedding of bounded
lattices (an $\Llat$\--embedding for short) of $A$ into
$\cO^\uparrow(\Spec^\uparrow A)$. If moreover $A$ is a Heyting algebra, it is an
embedding of Heyting algebras (an $\LHA$\--embedding for short). In
particular, for every $a,b\in A$ and every $x\in\Spec^\uparrow A$,
\begin{displaymath}
  a\to b\in x \iff \forall y\geq x,\ \big[a\in y \Rightarrow b\in y\big].
\end{displaymath}

Facts~\ref{fa:radical} and \ref{fa:ideal-quo} below are
folklore. The reader may refer in particular to
\cite{balb-dwin-1974}, chapters III and IV.

\begin{fact}\label{fa:radical}
  Every filter is the intersection of the prime filters which contain
  it. 
\end{fact}

Given an $\Llat$\--morphism $f:B\to A$ we let $\fKer f=f^{-1}(\UN)$
denote its {\df filter kernel}, and $f^\uparrow:\Spec^\uparrow A\to\Spec^\uparrow B$ be the
dual map defined by 
\begin{displaymath}
  f^\uparrow(x):=f^{-1}(x). 
\end{displaymath}

\begin{fact}\label{fa:ideal-quo}
  If $f$ is surjective then $f^\uparrow$ induces an increasing bijection 
  from $\Spec^\uparrow A$ to $\{y\in\Spec^\uparrow B\tq \fKer f\subseteq y\}$.
%   increasing bijections:
%   \begin{itemize}
%     \item
%       from $\Spec^\uparrow A$ to $\{y\in\Spec^\uparrow B\tq \fKer f\subseteq y\}$, and;
%     \item 
%       from $\Spec_\downarrow A$ to\footnote{This item follows from the
%         previous one, using that a subset of $L$ is a prime ideal
%         if and only if its complement is a prime filter.} 
%       $\{z\in\Spec_\downarrow B\tq \fKer f\cap z=\emptyset\}$.
%   \end{itemize}
\end{fact}

In Heyting algebras, every congruence comes from a filter. So the
classical ``Isomorphisms Theorems'' of universal algebra can be
expressed for morphisms of Heyting algebras ($\LHA$\--morphisms for
short) in terms of filters, just as analogous statements on
morphisms of rings are expressed in terms of ideals. In particular,
we will make use of the next property. 

\begin{fact}[Factorisation of Morphisms]\label{fa:facto-morph}
  Let $f:B\to A$ be an $\LHA$\--mor\-phism. Given a filter $I$ of $B$,
  let $\pi_I:B\to B/I$ be the canonical projection. Then $f$ factors
  through $\pi_I$, that is $f=g\circ\pi_I$ for some $\LHA$\--morphism
  $g:B/I\to A$, if and only if $\fKer f\subseteq I$. 
  \begin{center}
    \begin{tikzcd}
      B \arrow[r,"f"] \arrow[d,two heads,"\pi_I"'] & A \\
      B/I \arrow[ur,dashed,"g"'] & 
    \end{tikzcd}
  \end{center}
  In that case, such an $\LHA$\--morphism $g$ is unique. It is
  injective if and only if $\fKer f=I$.
\end{fact}

\paragraph{Esakia spaces and free Heyting algebras.}
An {\df Esakia space} (see \cite{esak-1974}) is a poset $X$ with a
compact topology such that: $(i)$ whenever $x \nleq y$ in $(X,\leq)$ there
is a clopen $U\subseteq X$ that is an up-set for $\leq$ such that $x\in U$ and
$y\notin U$; and $(ii)$ $U\downarrow$ is clopen whenever $U$ is a clopen subset of
$X$. Given a Heyting algebra $A$, its prime filter spectrum ordered
by inclusion and endowed with the patch topology is an Esakia space,
called the {\df Esakia dual} of $A$. We denote it by $\ES A$. 

Let $\bar p=\{p_1,\dots,p_l\}$ be a finite set of variables, $\Freebar p$
be the free Heyting algebra on $\bar p$, and $\ESFbar p$ its Esakia
space. For every $l$\--tuple $\bar a$ in a Heyting
algebra $A$ we let
\begin{displaymath}
  \pi_{\bar a}:\Freebar p\to A 
\end{displaymath}
denote the canonical morphism mapping $\bar p$ to $\bar a$. 

Every formula $\varphi(\bar p)$ in the Intuitionistic Propositional
Calculus (IPC) can be read as an $\LHA$\--term in the variables
$\bar p$, and conversely. The {\df degree} $\deg\varphi$ is the maximum
number of nested occurrences of $\to$ in $\varphi$. Abusing the notation, we
will also consider $\varphi$ as an element of $\Freebar p$. We let
\begin{displaymath}
  \oc\varphi(\bar p)\fc = \big\{x\in\ESFbar p\tq \varphi(\bar p)\in x\big\}.
\end{displaymath}
As a set of prime filters, $\oc\varphi(\bar p)\fc$ is nothing but
$P^\uparrow(\varphi(\bar p))$. The different notation simply reminds us to consider
it as a subset of $\ESFbar p$ with the patch topology, instead of a
subset of $\Spec^\uparrow \Freebar p$ with the Zariski topology.

\paragraph{Systems of equations and the Finite Extension Property.}
Every finite system of equations and negated equations in the
variables $(\bar p,\bar q)$, is equivalent to an
$\LHA$\--formula of the form
\begin{displaymath}
  \EQparsol p {\bar q} 
  =\Big[t(\bar p,\bar q)=\UN\land\lland_{k\leq \kappa}s_k(\bar p,\bar q)\neq\UN\Big].
\end{displaymath}
The {\df degree} of $\EQparsol p {\bar q}$, with $\bar s=(s_1,\dots,s_\kappa)$,
is the maximum of the degrees of the IPC formulas $t$ and $s_k$ for
$1\leq k\leq\kappa$. We will identify every such formula with the corresponding
system of equations and negated equations in $(\bar p,\bar q)$,
calling $\EQparsol p {\bar q}$ itself a {\df system of degree $d$}.
Given a tuple $\bar a$ in a Heyting algebra $H$, a {\df solution} of
$\EQpar a$ is then a tuple $\bar b$ such that $H\models\EQparsol a {\bar
b}$.

\begin{theorem}[Finite Extension Property]\label{th:FEP-intro}
  Let $\cV$ be a variety of Heyting algebras having the Finite Model
  Property. If an existential $\LHA$\--formula with parameters in a
  finite $\cV$\--algebra $A$ is satisfied in a $\cV$\--algebra
  containing $A$, then it is satisfied in some finite $\cV$\-algebra
  containing $A$. 
\end{theorem}

We will use this improvement of the classical Finite Model Property
only for the variety of all Heyting algebras, and for formulas $\exists\bar
q,\,\EQparsol a {\bar q}$. The proof requires some notions of
co-dimension coming from \cite{darn-junk-2011}. Since this material is
not needed anywhere else in this paper, we delay the proof of the
Finite Extension Property to the appendix in Section~\ref{se:appendix}.

\paragraph{Splitting and Density axioms.}

In \cite{darn-junk-2018} we introduced for co-Heyting algebras the
notion of ``strong order''. It corresponds in Heyting algebras to
the following relation. 
\begin{displaymath}
  b \gg a \iff b\to a=a\mbox{ and }b\geq a
\end{displaymath}
This is a strict order on $A\setminus\{\UN\}$ (not on $A$ because $\UN\gg a$ for
every $a\in A$, including $a=\UN$). Above all, we introduced the next
two axioms (denoted D1 and S1 in \cite{darn-junk-2018}), which
translate to the following ones in Heyting algebras.

\begin{description}
  \item{\bf [Density]} 
    For every $a,c$ such that $c \gg a \neq \UN$ there exists an
    element $b\neq\UN$ such that:
     \begin{displaymath}
       c \gg b \gg a 
     \end{displaymath}
  \item{\bf [Splitting]} 
    For every $a,b_1,b_2$ such that $b_1 \meet b_2 \gg a \neq \UN$ there
    exists $a_1\neq\UN$ and $a_2\neq\UN$ such that:
    \begin{displaymath}
      \begin{array}{c} 
        a_2\to a = a_1 \leq b_1\\
        a_1\to a = a_2 \leq b_2\\
        a_1 \join a_2 = b_1 \join b_2 
      \end{array}
    \end{displaymath}
\end{description}

The Density axiom clearly says that $\gg$ is a dense order. The
Splitting axiom implies the lack of atoms and co-atoms, but it is
actually much stronger (and more subtle) than this. For a geometric
interpretation of the Splitting Property, we refer the reader to
\cite{darn-junk-2018} and \cite{darn-2018-tmp}. The two next properties are theorems~1.1
and 1.2 in \cite{darn-junk-2018}. 

\begin{theorem}\label{th:E-clos-DS}
  Every existentially closed Heyting algebra has the Density
  and Splitting Properties. 
\end{theorem}

\begin{theorem}\label{th:DS-embed}
  Let $H$ be a Heyting algebra having the
  Density and 
  Splitting properties. Let $A$ be a finite
  Heyting subalgebra of $H$, and $B$ an extension of $A$. If $B$ is
  finite then $B$ embeds into $H$ over\footnote{With other words,
    there is an $\LHA$\--embedding of $B$ into $H$ which fixes $A$
  pointwise.} $A$. 
\end{theorem}

In particular, if $H$ is a non-trivial Heyting algebra (that is
$\ZERO\neq\UN$ in $H$) having the Density and Splitting properties, then
every finite non-trivial Heyting algebra embeds into $H$: it
suffices to apply Theorem~\ref{th:DS-embed} to $H$, $B$ and
$A=\{\ZERO,\UN\}$. Using the model-theoretic Compactness Theorem, the
next property follows immediately. 

\begin{corollary}\label{co:DS-sat-embed}
  Let $H$ be a non-trivial Heyting algebra which has the Density and
  Splitting properties. If $H$ is $\omega$\--saturated then every every
  finitely generated non-trivial Heyting algebra embeds into $H$.
\end{corollary}

% Corollary~3.2 in \cite{darn-junk-2018} says (after reversing the
% order of co-Heyting algebras) that, given an extension $A\subseteq B$ of
% finite Heyting algebras, $B$ is a minimal extension of $A$
% if and only if $B$ has exactly one more meet-irreducible element
% than $A$. The same holds true for join-irreducible elements of $A$
% and $B$. Indeed, in finite Heyting algebras every filter and every
% ideal is principal. So there are exactly the same number of
% join-irreducible and meet-irreducible elements (because the
% complement of a prime filter is a prime ideal and reciprocally). The
% next fact follows immediately. 

% \begin{fact}\label{fa:min-ext}
%   Let $B$ be a finite Heyting algebra with exactly $n+1$
%   join-irreducible elements, for some integer $n\geq0$. Then $B$
%   contains a Heyting algebra with exactly $n$ join-irreducible
%   elements. 
% \end{fact}

\paragraph{The Open Mapping Theorem.}
Following \cite{gool-regg-2018}, for every integer $d\geq0$ and every
$x,y\in\ESFbar p$ we write $x \sim_d y$ if: for every IPC formula $t(\bar
p)$ of degree $\leq d$, $t(\bar p)\in x \iff t(\bar p)\in y$. The relations
$\sim_d$ define on $\ESFbar p$ an ultrametric distance 
\begin{displaymath}
  \ddbar p(x,y)=2^{-\min\{d\tq x \not\sim_d y\}},
\end{displaymath}
with the conventions that $\min\emptyset=+\infty$ and $2^{-\infty}=0$. The topology
induced on $\ESFbar p$ by this distance is exactly the patch
topology (Lemma~8 in \cite{gool-regg-2018}). We let
\begin{displaymath}
  B_{\bar p}(x,2^{-d})=\big\{y\in\ESFbar p\tq \ddbar p(x,y)<2^{-d}\big\}
\end{displaymath}
denote the {\df ball of center $x$ and radius $2^{-d}$}. Note that
it is clopen and that
\begin{displaymath}
  y\in B_{\bar p}(x,2^{-d}) \iff y\sim_d x.
\end{displaymath}
So the balls of radius $2^{-d}$ in $\ESFbar p$ are the equivalence
classes of $\sim_d$. Recall that, up to intuitionist equivalence, there
are only finitely many IPC formulas in $\bar p$ of degree $<d$. In
particular, for every ball $B$ of radius $2^{-d}$ there is a pair
$(\varphi_B,\psi_B)$ of IPC formulas of degree $<d$ such that 
\begin{displaymath}
  B=\oc\varphi_B\fc\setminus\oc\psi_B\fc. 
\end{displaymath}
We let $\cB_d(\bar p)$ be the set of these balls, and fix for each
of them a pair $(\varphi_B,\psi_B)$ as above. 
% For each ball $B$ in $\ESFbar p$ 
% we fix once and for all a pair $(\varphi_B,\psi_B)$ as above. 

\begin{claim}\label{cl:phiB-y-Ker}
  Let $\bar a$ be an $l$\--tuple in some Heyting algebra $A$. For
  every ball $B$ in $\ESFbar p$, there is $y\in B$ such that
  $\fKer\pi_{\bar a}\subseteq y$ if and only if $\varphi_B(\bar a)\to\psi_B(\bar a)\neq\UN$.
  
\end{claim}

\begin{proof}
  Replacing $A$ if necessary by the Heyting algebra generated in $A$ by
  $\bar a$, we can assume that $\pi_{\bar a}$ is surjective. 
  Note first that for every IPC formula $\theta(\bar b)$ and every
  $y\in\ESFbar p$ containing $\fKer\pi_{\bar a}$,
  \begin{displaymath}
    \theta(\bar p)\in y\iff \theta(\bar a)\in\pi_{\bar a}(y).
  \end{displaymath}
  In particular, for every ball $B=\oc\varphi_B\fc\setminus\oc\psi_B\fc$ in $\ESFbar
  p$, there is $y\in B$ which contains $\fKer\pi_{\bar a}$ if and only
  if there is $y_A\in\ES A$ such that $\varphi_B(\bar a)\in y_A$ and $\psi_B(\bar
  a)\notin y_A$ (here we use Fact~\ref{fa:ideal-quo} and the assumption
  that $\pi_{\bar a}$ is surjective). By Stone's duality, the latter
  is equivalent to $\varphi_B(\bar a)\nleq\psi_B(\bar a)$, that is $\varphi_B(\bar
  a)\to\psi_B(\bar a)\neq\UN$. 

  With other words, $\varphi_B(\bar a)\to\psi_B(\bar a)=\UN$ if and only if
  there is no $y\in B$ containing $\fKer\pi_{\bar a}$, which proves our
  claim.
\end{proof}

The main result of \cite{gool-regg-2018} is that, for every pair of
finitely presented Heyting algebras $A$, $B$ and every
$\LHA$\--morphism $f:A\to B$, the dual map $f^{\uparrow}:\ES B\to \ES A$ is an
open map. We will focus on the special case where $f$ is the natural
inclusion from $A=\Freebar p$, to $B=\Free{\bar p,q}$, with $\bar p$
an $l$\--tuple of variables an $q$ a new variable.
Then $f^\uparrow(x)=x\cap\Freebar p$, for every $x\in\ESF{\bar p,q}$. In that
case, by compactness of $\ESFbar p$ the Open Mapping Theorem rephrases
as the following property. 

\begin{theorem}[Open Mapping for $\ESFbar p$]\label{th:RG}
  For every integer $d\geq0$, there is an integer
  $R_{l,d}\geq0$ such that 
  \begin{align*}
    \forall y\in \ESFbar p,\, & \forall x_0\in \ESF{\bar p,q}\,\Big[
      \ddbar p\big(y,x_0\cap\Freebar p\big)< 2^{-R_{l,d}} \nonumber \\
    & \Longrightarrow \exists x\in \ESF{\bar p,q}\,
      \Big(x\cap\Freebar p = y
        \mbox{ and }\dd{\bar p,q}(x,x_0)< 2^{-d}
      \Big)\Big]
  \end{align*}
\end{theorem}
It is actually from this more precise statement that the Open Mapping
Theorem for general finitely presented Heyting algebras is derived in
\cite{gool-regg-2018}. And it is this one that we will prove in
Section~\ref{se:OMT}.

\section{The prime existentially closed model}
\label{se:prime-model}

We are going to amalgamate a set of representatives of all the finite
Heyting algebras in a countable, locally finite Heyting algebra $\HH$
which will prove to be the smallest non-trivial existentially closed Heyting
algebra (up to isomorphism), and the only one which is both countable
and locally finite. The construction goes by induction.

\paragraph{Construction of $\HH$.}

Let $\widetilde{H}$ be a non-trivial existentially closed Heyting algebra. Let
$H_0$ be the two-element Heyting algebra contained in $\widetilde{H}$. Assume
that a finite Heyting algebra $H_n\subseteq\widetilde{H}$ has been constructed for
some $n\geq0$. Let $A_1,\dots,A_\alpha$ be the list of all the subalgebras of
$H_n$ with $\leq n$ join-irreducible elements. For each $i\leq\alpha$ let
$B_{i,1},\dots,B_{i,\beta_i}$ be a set of representatives of the minimal
proper finite extensions of $A_i$ (up to $\LHA$\--isomorphism over
$A_i$). Let $H_{n,1,1}=H_n$ and assume that for some $i_0\leq\alpha$ and
$j_0\leq\beta_{i_0}$, we have constructed a finite Heyting algebra
$H_{n,i_0,j_0}\subseteq\widetilde{H}$ containing: firstly $H_n$; secondly, for every
$i< i_0$ and $j\leq j_i$ a copy of $B_{i,j}$ containing $A_i$; and
thirdly, for every $j<j_0$ a copy of $B_{i_0,j}$ containing
$A_{i_0}$. This assumption is obviously true if $i_0=j_0=1$. In the
induction step we distinguish three cases. \medskip

{\it Case 1}: $j_0<\beta_{i_0}$. It is well known that the variety of
Heyting algebras has the Finite Amalgamation Property, that is every
pair of finite extensions of a given Heyting algebra $L$ amalgamate
over $L$ in a finite Heyting algebra. In particular, there is a
finite Heyting algebra $C$ containing both $H_{n,i_0,j_0}$ and a
copy $B$ of $B_{i_0,j_0+1}$. By Theorem~\ref{th:E-clos-DS}, $H$
satisfies the Spitting and Density axioms. So by
Theorem~\ref{th:DS-embed}, there is an $\LHA$\--embedding $\varphi:C\to H$
which fixes $H_{n,i_0,j_0}$ pointwise. We let $H_{n,i_0,j_0+1}$ be
the Heyting algebra generated in $H$ by the union of $H_{n,i_0,j_0}$
and $\varphi(B)$. This is a finite Heyting algebra containing $H_n$ and,
for every $i< i_0$ and $j\leq j_i$ a copy of $B_{i,j}$ containing
$A_i$, and for every $j<j_0+1$ a copy of $B_{i_0,j}$ containing
$A_{i_0}$ (namely $\varphi(B_{i_0,j_0+1}$ for $j=j_0+1$, the other cases
being covered by the induction hypothesis because $H_{n,i_0,j_0+1}$
contains $H_{n,i_0,j_0}$).
\medskip

{\it Case 2}: $j_0=\beta_{i_0}$ and $i_0<\alpha$. By the Finite Amalgamation Property
again, there is a finite Heyting algebra $C$ containing both
$H_{n,i_0,j_0}$ and a copy $B$ of $B_{i_0+1,1}$.
Theorems~\ref{th:E-clos-DS} and \ref{th:DS-embed} then give an
$\LHA$\--embedding $\varphi:C\to\widetilde{H}$ which fixes $H_{n,i_0,j_0}$ pointwise.
We let $H_{n,i_0+1,1}$ be the Heyting algebra generated in $\widetilde{H}$ by
the union of $H_{n,i_0,j_0}$ and $\varphi(B)$. This is a finite Heyting
algebra containing $H_n$ and, for every $i< i_0+1$ and $j\leq j_i$ a
copy of $B_{i,j}$ containing $A_i$, and a copy of $B_{i_0,1}$
containing $A_{i_0}$. 
\medskip

{\it Case 3}: $j_0=\beta_{i_0}$ and $i_0=\alpha$. We let
$H_{n+1}=H_{n,\alpha,\beta_\alpha}$. 
\medskip

Repeating this construction, we get an increasing chain of finite
Heyting algebras $H_n$ contained in $\widetilde{H}$. We let $\HH=\bigcup_{n\in\NN}H_n$.
Since every $H_n$ is finite, $\HH$ is a countable and locally finite
Heyting algebra. 

\begin{proposition}\label{pr:H0-DS}
  Given a Heyting algebra $H$, the following properties are
  equivalent. 
  \begin{enumerate}
    \item\label{it:H0-embed}
      $H$ is countable, locally finite, and every finite
      extension $B$ of a subalgebra $A$ of $H$ embeds into $H$ over
      $A$.
    \item\label{it:H0-DS}
      $H$ is countable, locally finite, and has the Density
      and Splitting properties. 
    \item\label{it:H0-EC}
      $H$ is countable, locally finite and existentially closed.
    \item\label{it:H0-isom}
      $H$ is $\LHA$\--isomorphic to $\HH$. 
  \end{enumerate}
\end{proposition}

\begin{proof}
  (\ref{it:H0-isom})$\Rightarrow$(\ref{it:H0-embed}) We have to prove that
  $\HH$ itself has the properties of item~\ref{it:H0-embed}. We
  already know that $\HH$ is countable and locally finite. Let $B$
  be a finite extension of a subalgebra $A$ of $H$. We want to embed
  $B$ in $H$ over $A$. By an immediate induction we can assume that
  $B$ is a minimal proper finite extension of $A$. Now $A$ is
  contained in $H_n$ for some integer $n\geq0$, which we can assume to
  be greater than the number of join-irreducible elements of $A$. So
  $A$ is one of the algebras $A_1,\dots,A_\alpha$ of the above construction
  of $\HH$, and $B$ is isomorphic over $A$ to some $B_{i,j}$ as
  above. By construction $H_{n+1}$ contains a copy of $B_{i,j}$,
  hence of $B$, which contains $A$. So $B$ embeds over $A$ into
  $H_{n+1}$, hence {\it a fortiori} into $\HH$. 

  (\ref{it:H0-embed})$\Rightarrow$(\ref{it:H0-isom}) Both $H$ and $\HH$ have
  the properties of item~\ref{it:H0-embed}. We construct an
  $\LHA$\--isomorphism by induction. Let $(a_n)_{n\in\NN}$ be an
  enumeration of $H$, and $(b_n)_{n\in\NN}$ an enumeration of $\HH$.
  Assume that for some integer $n\geq0$ we have constructed an
  $\LHA$\--isomorphism $\varphi_n:A_n\to B_n$ where $A_n$ (resp. $B_n)$ is a
  finite subalgebra of $H$ (resp. $\HH$) containing $a_k$ (resp.
  $b_k$) for every $k< n$. This assumption is obviously true for
  $n=0$, with $\varphi_0$ the empty map. Let $k_0$ be the first integer
  $k\geq n$ such that $a_k\notin A_n$. Let $A_{n+1/2}$ be the Heyting
  algebra generated in $H$ by $A_n\cup\{a_{k_0}\}$. Since $H$ is locally
  finite, $A_{n+1/2}$ is a finite extension of $A_n$. Composing
  $\varphi_n^{-1}$ by the inclusion map, we get an embedding of $B_k$ into
  $A_{n+1}$. The assumption (\ref{it:H0-embed}) gives an embedding
  of $A_{n+1}$ into $\HH$ over $B_n$, or equivalently a subalgebra
  $B_{n+1/2}$ of $\HH$ containing $B_n$ and an $\LHA$\--isomorphism
  $\varphi_{n+1/2}:A_{n+1/2}\to B_{n+1/2}$ whose restriction to $A_n$ is
  $\varphi_n$. Now let $l_0$ be the smallest integer $l\geq n$ such that
  $b_l\notin B_{l+1/2}$, and let $B_{l+1}$ be the Heyting algebra
  generated in $\HH$ by $B_{n+1/2}\cup\{b_{l_0}\}$. A symmetric argument
  gives a subalgebra $A_n$ of $H$ containing $A_{n+1/2}$ and an
  $\LHA$\--isomorphism $\psi_{n+1}:B_{n+1}\to A_{n+1}$ whose restriction
  to $A_{n+1/2}$ is $\varphi_{n+1/2}^{-1}$. We then let
  $\varphi_{n+1}=\psi_{n+1}^{-1}$ and let the induction go on. Finally we get
  an isomorphism $\varphi:H\to\HH$ whose restriction to each $A_n$ is $\varphi_n$,
  and we are done. 

%   (\ref{it:H0-embed})$\Rightarrow$(\ref{it:H0-DS}) Let $a,c\in H$ be such that
%   $\UN\neq a\gg c$. By assumption (\ref{it:H0-embed}) $H$ is locally
%   finite, so the Heyting algebra $A$ generated by $a$ and $c$ in $H$
%   is finite. By Lemma~4.1 in \cite{darn-junk-2018} there is a finite
%   Heyting algebra $B$ containing $A$, and there is an element
%   $b\neq\UN$ in $H$ such that $a\gg b\gg c$. By assumption
%   (\ref{it:H0-embed}) there is an $\LHA$\--embedding $\varphi:B\to H$ which
%   fixes $A$ pointwise. This embedding preserves the $\gg$ relation
%   (because it is quantifier-free definable in $\LHA$) hence $a\gg\varphi(b)\gg
%   c$, and $\varphi(b)\neq\UN$ because $\varphi$ is injective, hence $H$ satisfies
%   the Density axiom. 

%   Now let $a,b_1,b_2\in H$ such that $\UN\neq a\gg b_1\join b_2$. Again the
%   Heyting algebra $A'$ generated in $H$ by these elements is finite.
%   By Lemma~4.2 in \cite{darn-junk-2018} there is a finite Heyting
%   algebra $B'$ containing $A'$, and there are elements $a_1$, $a_2$
%   in $B'\setminus\{\UN\}$ such that $b_1\geq a_1=a_2\to a$, $b_2\geq a_2=a_1\to a$, and
%   $b_1\join b_2=a_1\join a_2$. All these properties of $a_1$ and $a_2$ are
%   preserved by $\LHA$\--embedding, provided $a_1$, $a_2$ and $a$ are
%   fixed. The assumption (\ref{it:H0-embed}) provides such an
%   embedding of $B'$ into $H$ over $A'$, which proves that $H$
%   has the Splitting Property. 

  (\ref{it:H0-embed})$\Rightarrow$(\ref{it:H0-EC})
  Let $\varphi(\bar a,\bar q)$ be a quantifier-free $\LHA$\--formula with
  parameters $\bar a\in\HH^l$, such that $\exists\bar q\,\varphi(\bar a,\bar q)$
  is satisfied in some extension of $\HH$. Let $A$ be the Heyting
  algebra generated by $\bar a$. Since $\HH$ is locally finite, $A$
  is finite. By the Finite Extension Property, it follows that
  $\exists\bar q\,\varphi(\bar a,\bar q)$ is satisfied in some finite extension
  $B$ of $A$. By assumption (\ref{it:H0-embed}), $B$ embeds into $\HH$
  over $A$. So $\exists\bar q\,\varphi(\bar a,\bar q)$ is satisfied in $\HH$.

  (\ref{it:H0-EC})$\Rightarrow$(\ref{it:H0-DS})$\Rightarrow$(\ref{it:H0-embed}) is just
  Theorem~\ref{th:E-clos-DS} and Theorem~\ref{th:DS-embed}, which
  finishes the proof. 
\end{proof}

\begin{remark}\label{re:H0-prime}
  We have constructed $\HH$ inside a given non-trivial existentially closed
  Heyting algebra $\widetilde{H}$. The same construction in another
  non-trivial existentially closed Heyting algebra $\widetilde{H}'$ would
  give a Heyting algebra $\HH'$ contained in $\widetilde{H}'$ to which
  Proposition~\ref{pr:H0-DS} would apply as well as to $\HH$. The
  characterisation given there proves that $\HH$ and $\HH'$ are
  isomorphic. As a consequence, every non-trivial existentially closed
  Heyting algebra contains a copy of $\HH$. 
\end{remark}

\section{The QE property}
\label{se:hensel}

Given an integer $n\geq0$ and $l$\--tuples $\bar a$, $\bar a'$ in a
Heyting algebra $A$ let
\begin{displaymath}
  \Th_n(\bar a)=\{\varphi(\bar p)\tq \deg \varphi\leq n\mbox{ and }\varphi(\bar a)=\UN\} 
\end{displaymath}
and
\begin{displaymath}
  Y_n(\bar a)= \big\{y\in\ESFbar p\tq 
    \exists y'\in\ESFbar p,\ y\sim_n y'\mbox{ and }\fKer\pi_{\bar a}\subseteq y'\big\}. 
\end{displaymath}

We say that $\bar a$ and $\bar a'$ are {\df $n$\--similar} if
$\Th_n(\bar a)=\Th_n(\bar a')$, and write it $\bar a\approx_n \bar a'$.
They are {\df $\omega$\--similar} if they are $n$\--similar for
every $n$, or equivalently if $\fKer\pi_{\bar a}=\fKer\pi_{\bar a'}$.
Note that $\bar a$ and $\bar a'$ are $\omega$\--similar if and only if
the function mapping $a_i$ to $a'_i$ for $1\leq i\leq l$, extends to an
$\LHA$\--isomorphism between the Heyting algebras generated by $\bar
a$ and $\bar a'$ in $A$.

\begin{proposition}\label{pr:n-equiv}
  Let $\bar a$, $\bar a'$ be two $l$\--tuples in a Heyting algebra
  $A$. For every integer $n\geq0$ the following assertions are
  equivalent. 
  \begin{enumerate}
%     \item\label{it:om-equiv}
%       $\bar a$ and $\bar a'$ are $\omega$\--similar.
    \item\label{it:n-equiv-Th}
      $\Th_n(\bar a)=\Th_n(\bar a')$.
    \item\label{it:n-equiv-Y}
      $Y_n(\bar a)=Y_n(\bar a')$. 
    \item\label{it:n-equiv-B}
      For every ball $B$ in $\ESFbar p$ with radius $2^{-n}$:
      $\varphi_B(\bar a)\leq\psi_B(\bar a)$ if and only if $\varphi_B(\bar
      a')\leq\psi_B(\bar a')$.
  \end{enumerate}
\end{proposition}

\begin{proof}
  (\ref{it:n-equiv-Y})$\Leftrightarrow$(\ref{it:n-equiv-B}) Given $y\in \ESFbar p$,
  let $B$ the ball in $\ESFbar p$ with center $y$ and radius
  $2^{-n}$. By Claim~\ref{cl:phiB-y-Ker}, $y\in Y_n(\bar a)$ if and
  only if $\varphi_B(\bar a)\to\psi_B(\bar a)\neq\UN$, that is if and only if
  $\varphi_B(\bar a)\nleq \psi_B(\bar a)$. The equivalence follows.

  (\ref{it:n-equiv-Th})$\Rightarrow$(\ref{it:n-equiv-B}) Given a ball $B$ in
  $\ESFbar p$ with radius $2^{-n}$, $\varphi_B$ and $\psi_B$ have degree $<n$
  hence $\varphi_b\to\psi_B$ have degree $\leq n$. The assumption
  (\ref{it:n-equiv-Th}) then implies that $\varphi_B(\bar a)\to\psi_B(\bar
  a)=\UN$ if and only if $\varphi_B(\bar a')\to\psi_B(\bar a')=\UN$, and
  (\ref{it:n-equiv-B}) follows.

  (\ref{it:n-equiv-B})$\Rightarrow$(\ref{it:n-equiv-Y}) Assume that
  $\Th_n(\bar a)\neq\Th_n(\bar a')$, for example $\Th_(\bar
  a)\not\subseteq\Th_n(\bar a')$, and pick any $\theta(\bar p)$ in $\Th_n(\bar
  a')\setminus\Th_n(\bar a)$. Then $\theta$ belongs to $\fKer\pi_{\bar a'}$
  but not to $\fKer\pi_{\bar a}$. Fact~\ref{fa:radical} then gives
  $y\in\ESFbar p$ such that $\fKer\pi_{\bar a}\subseteq y$ and $\theta\notin y$. In
  particular $y\in Y_n(a)$, but for every $y'\in\ESFbar a$ such that
  $y\sim_n y'$, since $\theta(\bar p)$ has degree $\leq n$  we have
  \begin{displaymath}
    \theta\notin y \Longrightarrow \theta\notin y' \Longrightarrow \fKer\pi_{\bar a'}\not\subseteq y', 
  \end{displaymath}
  so $y\notin Y_n(\bar a')$. 
\end{proof}

\begin{remark}\label{re:EQUIV}
  Item~\ref{it:n-equiv-B} of Proposition~\ref{pr:n-equiv} gives that
  the relation $\approx_n$ for $l$\--tuples in $A$ is definable by the
  $\LHA$\--formula with $2l$ free variables
  \begin{displaymath}
    \EQUIV_{l,n}(\bar p,\bar p')=\lland_{B\in\cB_n(\bar p)}
       \Big[\big(\varphi_B(\bar p)\leq\psi_B(\bar p)\big)
       \Leftrightarrow\big(\varphi_B(\bar p')\leq\psi_B(\bar p')\big)\Big].
  \end{displaymath}
\end{remark}

Now we can introduce the corner stone of our axiomatization of
existentially closed Heyting algebras. Given a Heyting algebra $A$, we
let $h_{l,d}(A)$ be the smallest integer $n$ such that for every $\bar
a,\bar a'\in A^l$ such that $\bar a\approx_n\bar a'$: for every system $\EQpar
p$ of degree $\leq d$, $\EQparsol a q$ has a solution in $A$ if and only
if $\EQparsol{a'} q$ has a solution in $A$. If no such integer exists
we let $h_{l,d}(A)=+\infty$. We call $h_{l,d}(A)$ the {\df $(l,d)$\--index}
of $A$. We say that $A$ has the {\df QE property} if $h_{l,d}(A)$
is finite for every $(l,d)$. The following characterisation motivates
this terminology. 

\begin{example}\label{ex:dense-clos}
  If $T$ is any theory of Heyting algebras which eliminates the
  quantifier in $\LHA$, then every model of $T$ has the QE property by
  Proposition~\ref{pr:QE} below. In particular:
  \begin{itemize}
    \item
      Every dense Boolean algebra has the QE property. 
    \item 
      By Pitt's result every existentially closed Heyting algebra 
      has the QE property.
  \end{itemize}
\end{example}

\begin{proposition}\label{pr:QE}
  A Heyting algebra $A$ has the QE property if and only if the
  complete theory of $A$ in $\LHA$ eliminates the quantifiers. 
\end{proposition}

\begin{proof}
  Assume that the theory of $A$ in $\LHA$ eliminates the quantifiers.
  There are up to intuitionist equivalence finitely many IPC formulas
  in $l$ variables with degree $\leq d$, hence finitely many systems in
  $l+1$ variable of degree $\leq d$. Let $\cS_i(\bar p,q)$ for $1\leq i\leq N$
  an enumeration of them. For each of them, by assumption there is a
  quantifier-free formula $\chi_{\cS}(\bar p)$ such that the theory of
  $A$ proves that $\exists q,\,\cS_i(\bar p,q)$ is equivalent to $\chi_i(\bar
  p)$. Let $n$ be the maximal degree of all the atomic formulas
  composing these formulas $\chi_i$. For every $\bar a,\bar a'\in A^n$ such
  that $\bar a\approx_n\bar a'$, $A\models \chi_i(\bar a)$ if and only if $A\models\chi_i(\bar
  a')$. So the system $\cS_i(\bar a,q)$ has a solution in $A$ if and
  only if the system $\cS_i(\bar a',q)$ has a solution in $A$, that is
  $h_{l,d}(A)\leq n$. 

  Reciprocally, assume that $h_{l,d}(A)$ is finite for every integers
  $l$, $d$. Let $\varphi(\bar p,a)$ be a quantifier-free formula in $l+1$
  variables. It is logically equivalent to a disjunction of finitely
  many systems $\cS_i(\bar p,q)$, for $1\leq i\leq N$. Let $d$ be the
  maximal degree of all these systems, and $n=h_{l,d}(A)$. Using again
  that there are, up to intuitionist equivalence, finitely many IPC
  formulas in $l$ variables with degree $\leq n$, there are also finitely
  many equivalence classes for $\approx_n$ in $A^l$. For every such
  equivalence class $C$, let $\bar a\in C$ and let $\varphi_C$ (resp. $\psi_C$)
  be the conjunction (resp. disjunction) of the formulas $t(\bar
  p)=\UN$ (resp $t(\bar p)\neq\UN$) for $t(\bar p)$ ranging over a finite
  set of representatives of the IPC formulas which belong to
  $\Th_n(\bar a)$ (resp. which don't belong to $\Th_n(\bar a)$).
  Clearly a tuple $\bar a'\in A^l$ belongs to $C$ if and only if
  $A\models\varphi_C(\bar a)\land \lnot \psi_C(\bar a)$. Now, for each system $\cS_i(\bar
  p,q)$ let $\cC_i$ be the list of equivalence classes $C$ for $\approx_n$
  in $A^l$ such that $\cS_i(\bar a,q)$ has a solution in $A$ for some
  $\bar a\in C$ (hence also for every $\bar a\in C$ by definition of
  $n=h_{l,d}(A)$). By construction, for every $\bar a\in A^l$
  \begin{displaymath}
    A\models\exists q,\,\cS_i(\bar a,q) \iff  \bar a \in\bigcup\cC_i \iff A\models
    \llor_{C\in\cC_i}\Big(\varphi_C(\bar a)\land \lnot \psi_C(\bar a)\Big).
  \end{displaymath}
  Let $\chi_i(\bar p)$ be the IPC formula on the right. By construction
  the theory of $A$ proves that $\exists q,\,\varphi(\bar p,q)$ is equivalent to
  $\llor_{i\leq N}\chi_i(\bar p)$. The latter is quantifier-free, hence the
  theory of $A$ eliminates the quantifiers by an immediate induction. 
\end{proof}

\begin{proposition}\label{pr:axiom-QE}
  For any given $l,d,n\in\NN$, there is a closed $\forall\exists$\--formula
  $\FC_{l,d}^n$  in $\LHA$ for every Heyting algebras $A$ 
  \begin{displaymath}
    h_{l,d}(A)\leq n \iff A\models\FC_{l,d}^n. 
  \end{displaymath}
  In particular, every Heyting algebra elementarily equivalent to $A$
  has the same $(l,d)$\--index. 
\end{proposition}

\begin{proof}
  For every system $\EQpar p$ of degree $\leq d$ let 
  \begin{displaymath}
    \FC^n_{t,\bar s}(\bar p,\bar p')=
    \Big[\left(\EQUIV_{l,n}(\bar p,\bar q') 
      \land \exists q',\,\EQparsol{p'}{q'} \right) \Rightarrow \exists b,\,\EQpar p\Big]
  \end{displaymath}
  where $\EQUIV_{l,n}(\bar p,\bar p')$ is the quantifier-free
  $\LHA$\--formula introduced in Remark~\ref{re:EQUIV}. This
  is a $\forall\exists$\--formula. Clearly $h_{l,d}(A)\leq n$ if and only if
  $A\models\FC^n_{t,\bar s}(\bar a,\bar a')$ for every $\bar a,\bar a'\in
  A^l$ and every system $\EQpar p$ of degree $\leq d$. There are only
  finitely many non-equivalent IPC formulas of degree $\leq d$, hence
  finitely many non-equivalent systems of degree $\leq d$. So
  $h_{l,d}(A)\leq n$ if and only if $A$ satisfies the conjunction of
  the finitely many formulas $\forall\bar p,\bar p'\,\FC^n_{t,\bar s}(\bar
  p,\bar p')$. 
\end{proof}

Let $T$ be the model-completion of the theory of Heyting algebras. By
Pitt's result the models of $T$ are exactly the existentially closed
Heyting algebras. In particular, the Heyting algebra $\HH$ constructed
in Section~\ref{se:prime-model} is a model of $T$. So by
Proposition~\ref{pr:QE} $\HH$ has the QE property. Once and for all,
we let $h_{l,d}=h_{l,d}(\HH)<+\infty$. 

\begin{theorem}\label{th:axiom-mod-comp}
  For every Heyting algebra $H$ the following properties are
  equivalent. 
  \begin{enumerate}
    \item\label{it:axiom-EC}
      $H$ is existentially closed.
    \item\label{it:axiom-DSQE}
      $H$ has the Density, the Splitting and the QE properties. 
    \item\label{it:axiom-DSQE-H0}
      $H$ has the Density and Splitting properties, and
      $h_{l,d}(H)=h_{l,d}$ for every integers $l$, $d$. 
  \end{enumerate}
  In particular, the model-completion of the theory of Heyting
  algebras is axiomatised by the Density and Splitting axioms, and the
  formulas $\FC^{h_{l,d}}_{l,d}$  for every integers $l$, $d$.
\end{theorem}

\begin{proof}
  (\ref{it:axiom-EC})$\Rightarrow$(\ref{it:axiom-DSQE-H0}) Since $H$ is
  existentially closed, it has the Density and Splitting
  properties by Theorem~\ref{th:E-clos-DS}. Moreover $\HH\subseteq H$ by
  Remark~\ref{re:H0-prime}, and $\HH$ is existentially closed by
  Proposition~\ref{pr:H0-DS}, so $\HH\preccurlyeq H$  y
  Pitts' result. In particular $\HH\equiv H$ so $h_{l,d}(H)=h_{l,d}(\HH)$ for every
  $l$, $d$ by Proposition~\ref{pr:axiom-QE}.

  (\ref{it:axiom-DSQE-H0})$\Rightarrow$(\ref{it:axiom-DSQE}) is clear.

  (\ref{it:axiom-DSQE})$\Rightarrow$(\ref{it:axiom-EC}) 
  The case of the one-point Heyting algebra being trivial, we can
  assume that $\ZERO\neq\UN$ in $H$. By model-theoretic non-sense, we can
  assume that $H$ is $\omega$\--saturated and it suffices to prove that for
  every $\bar a\in H^l$ and every finitely generated Heyting algebra $B$
  containing the algebra $A$ generated by $\bar a$ in $H$, there is an
  embedding of $B$ into $H$ over $A$. By an immediate induction we are
  reduced to the case where $B$ is generated over $A$ by a single
  element $b$. Replacing if necessary $H$ by an elementary extension,
  thanks to Proposition~\ref{pr:axiom-QE} we can assume that $H$ is
  $\omega$\--saturated.

  Let $\Sigma$ be the set of atomic or neg-atomic $\LHA$\--formulas $\varphi(\bar
  p,q)$ such that $B\models\varphi(\bar a,b)$. We only have to prove that for some
  $b'\in H$, $\Sigma$ is satisfied by $(\bar a,b')$ (by which we mean that
  $H\models\varphi(\bar a,b')$ for every formula $\varphi(\bar p,q)$ in $\Sigma$). Indeed,
  $B$ will then be isomorphic over $A$ to the Heyting algebra
  generated in $H$ by $(\bar a,b')$, so we are done. Since $H$ is
  $\omega$\--saturated, it suffices to prove that every finite fragment of
  $\Sigma$ is satisfied by $(\bar a,b')$ for some $b'\in H$. Such a finite
  fragment is equivalent to a system. So we want to prove that
  for every system $\EQpar p$, if $\EQpar a$ has a solution in $B$
  then it has a solution in $H$. 

  Now let $\EQpar p$ be a system such that $\EQpar a$ has a solution in
  $B$, and let $d$ be the degree of $\EQpar p$. By
  Corollary~\ref{co:DS-sat-embed} $B$ embeds into $H$. So let $(\bar
  a_0,b_0)\in H^{l+1}$ be the image of $(\bar a,b)$ by such an
  embedding. Then $\Th_n(\bar a,b)=\Th_n(\bar a_0,b_0)$ for every
  integer $n\geq0$. {\it A fortiori} $\Th_n(\bar a)=\Th_n(\bar a_0)$ for
  every $n$, and in particular for $n=h_{l,d}(H)$. Now $\EQpar{a_0}$
  has solution in $H$, namely $b_0$, and $\bar a\approx_n\bar a_0$ for
  $n=h_{l,d}(H)$. So $\EQpar a$ has a solution in $H$, by definition
  of $h_{l,d}(H)$. 
\end{proof}

\section{The Open Mapping Theorem}
\label{se:OMT}

Our aim in this section is to derive the Open Mapping Theorem of
\cite{gool-regg-2018} from Theorem~\ref{th:axiom-mod-comp}.

We first need a lemma. As usually we consider IPC formulas $t(\bar
p,q)$ and $s_k(\bar p,q)$ as elements of $\Free{\bar p,q}$. In
particular, $t(\bar p,q)\uparrow$ in the next result is the principal filter
generated by $t(\bar p,q)$ in $\Free{\bar p,q}$.

\begin{lemma}\label{le:zero-Ker-Esa}
  Let $\EQpar p$ be a system, with $\bar s=(s_1,\dots,s_\kappa)$.
  For every $l$\--tuple $\bar a$ in a Heyting algebra $A$, the
  following assertions are equivalent. 
  \begin{enumerate}
    \item\label{it:zero}
      $\EQpar a$ has a solution in some extension of $A$.
    \item\label{it:Ker}
      $\fKer\pi_{\bar a}$ contains $t(\bar p,q)\uparrow\cap\Freebar p$, and the
      filter generated in $\Free{\bar p,q}$ by $\fKer\pi_{\bar
      a}\cup\{t(\bar p,q)\}$ does not contain any of the $s_k(\bar p,q)$'s.
    \item\label{it:Esa}
      For every $y\in\ESFbar p$ containing $\fKer \pi_{\bar a}$, there
      is $x\in \oc t(\bar p,q)\fc$ such that $x\cap\Freebar p =y$.
      Moreover, for each $k\leq \kappa$, there is at least one of these $y$'s
      for which $x$ can be chosen outside of $\oc s_k(\bar p,q)\fc$. 
  \end{enumerate}
\end{lemma}

\begin{proof}
Let $A_0$ be the Heyting algebra generated in $A$ by $\bar a$. If
$\EQpar a$ has a solution in some extension of $A_0$, by the
amalgamation property of Heyting algebras it has a solution in some
extension of $A$, and reciprocally. Thus, replacing $A$ by $A_0$ if
necessary, we can assume that $A$ is generated by $\bar a$. We let $G$
be the filter generated in $\Free{\bar p,q}$ by $\fKer\pi_{\bar a}\cup\{t\}$,
or equivalently by $\fKer\pi_{\bar a}\cup t\uparrow$. 

(\ref{it:zero})$\Rightarrow$(\ref{it:Ker}) In some extension $H$ of $A$, let
$b$ be such that $t(\bar a,b)=\UN$ and $s_k(\bar a,b)\neq\UN$ for every
$k\leq\kappa$. There is a
unique $\LHA$\--morphism $\pi_{\bar a,b}:\Free{\bar p,q}\to H$ mapping
$(\bar p,q)$ onto $(\bar a,b)$. By construction $\pi_{\bar a}$ is the
restriction of $\pi_{\bar a,b}$ to $\Freebar p$, hence $\fKer\pi_{\bar
a}=\fKer\pi_{\bar a,b}\cap\Freebar p$. We have
\begin{displaymath}
  t\in\fKer\pi_{\bar a,b}\ \Rightarrow \ t\uparrow\subseteq\fKer\pi_{\bar a,b} 
  \ \Rightarrow \ t\uparrow\cap\Freebar p \subseteq \fKer\pi_{\bar a,b}\cap\Freebar p.
\end{displaymath}
So $t\uparrow\cap\Freebar p$ is contained in $\fKer\pi_{\bar a}$.
Moreover $G$ is contained in $\fKer\pi_{\bar a,b}$ and
$s_k\notin\fKer\pi_{\bar a,b}$, hence $s_k\notin G$, for every $k\leq\kappa$.

(\ref{it:Ker})$\Rightarrow$(\ref{it:zero}) Assuming (\ref{it:Ker}), we are
claiming that $G\cap\ESFbar p=\fKer\pi_{\bar a}$. The right-to-left
inclusion is clear. Reciprocally, if $\varphi(\bar p)\in G$ then $\varphi(\bar
p)\geq\varphi_0(\bar p)\meet\tau(\bar p,q)$ for some $\varphi_0\in\fKer\pi_{\bar a}$ and
$\tau\in t\uparrow$. So $\varphi_0\to\varphi\geq\tau$, that is $\varphi_0\to\varphi\in t\uparrow\cap\ESFbar p$, hence
$\varphi_0\to\varphi\in\fKer\pi_{\bar a}$ by assumption (\ref{it:Ker}). It follows that
$\varphi_0\meet(\varphi_0\to\varphi)\in\fKer\pi_{\bar a}$, that is $\varphi_0\meet\varphi\in\fKer\pi_{\bar a}$. {\it A
fortiori} $\varphi\in\fKer\pi_{\bar a}$ which proves our claim. 

Let $B=\ESF{\bar p,q}/G$ and $f:\ESF{\bar p,q}\to B$ be
the canonical projection. $G\cap\ESFbar p$ is obviously the filter
kernel of the restriction of $f$ to $\ESFbar p$. Since $G\cap\ESFbar
p=\fKer\pi_{\bar a}$, by Factorisation of Morphisms we then have a
commutative diagram as follows (with left arrow the natural
inclusion).

\begin{center}
  \begin{tikzcd}
    \Free{\bar p,q} 
      \arrow[r,two heads,"f"]
    & B\vphantom{(} 
    \\
    \Freebar p
      \arrow[u,hook]
      \arrow[r,two heads,"\pi_{\bar a}"]
    & A
      \arrow[u,hook,dashed]
  \end{tikzcd}
\end{center}

We have $t\in G$ by construction, and $s_k\notin G$ for every $k\leq \kappa$ by assumption
(\ref{it:Ker}). So, after identification of $A$ with its image by
the dashed arrow, $B$ is an extension of $A$ in which $f(q)$ is a
solution of $\EQpar a$.

(\ref{it:Ker})$\Rightarrow$(\ref{it:Esa}) 
Given $y\in\ESFbar p$ containing $\fKer \pi_{\bar a}$, assume for a
contradiction that for every $x\in \oc t\fc$, $x\cap\Freebar p\neq y$. If
$x\cap\Freebar p\not\subseteq y$, pick any $\varphi_x$ in $(x\cap\Freebar p)\setminus y$ and let
$U_x=\oc\varphi_x\fc$. If $x\cap\Freebar p\not\supseteq y$, pick any $\psi_x$ in
$y\setminus(x\cap\Freebar p)$ and let $U_x=\oc\psi_x\fc^c$. The family
$(U_x)_{x\in\oc t\fc}$ is an open cover of the compact space $\oc
t\fc$. A finite subcover gives $\varphi_1,\dots,\varphi_\alpha\in \Freebar p\setminus y$ and
$\psi_1,\dots,\psi_\beta\in y$ such that 
\begin{displaymath}
  \oc t\fc \subseteq \bigcup_{i\leq\alpha}\oc\varphi_i\fc\cup \bigcup_{j\leq\beta}\oc\psi_j\fc^c.
\end{displaymath}
That is $\oc t\fc \subseteq \oc\varphi\fc\cup\oc\psi\fc^c$ with $\varphi=\jjoin_{i\leq\alpha}\varphi_i$ and
$\psi=\mmeet_{j\leq\beta}\psi_j$. In particular, $t\leq \psi\to\varphi$ so 
\begin{displaymath}
  \psi\to\varphi \in t\uparrow\cap\Freebar p \subseteq \fKer\pi_{\bar a}\subseteq y \Longrightarrow \psi\to\varphi\in y. 
\end{displaymath}
On the other hand every $\psi_i\in y$ hence $\psi\in y$, and every $\varphi_i\notin y$
hence $\varphi\notin y$ (because $y$ is a prime filter). In particular
$\psi\to\varphi\notin y$, a contradiction. 

We have proved that the first part of (\ref{it:Ker}) implies the first
part of (\ref{it:Esa}). Since $G$ is the intersection of the prime
filters of $\Free{\bar p,q}$ which contain $G$, the second part of
(\ref{it:Ker}) implies that for each $k\leq\kappa$ there is $x\in\ES{\bar
p,q}$ containing $\fKer\pi_{\bar a}$ and $t$ but not $s_k$. Letting
$y=x\cap\Freebar p$, we get the second part of (\ref{it:Esa}).

(\ref{it:Esa})$\Rightarrow$(\ref{it:Ker}) 
Let $\cY=\big\{y\in\ESFbar p\tq\fKer\pi_{\bar a}\subseteq y\big\}$ and
$\cX=\{x\cap\Freebar p\tq x\in\oc t\fc\}$. The first part of (\ref{it:Esa})
says that $\cY\subseteq\cX$, hence $\bigcap\cY\supseteq\bigcap\cX$. Since every filter is the
intersection of the prime filters which contain it,
$\bigcap\cY=\fKer\pi_{\bar a}$ and $\bigcap\oc t\fc=t\uparrow$ hence
\begin{displaymath}
  \bigcap\cX=\bigcap\oc t\fc\cap\Freebar p=t\uparrow\cap\Freebar p. 
\end{displaymath}
So $\fKer\pi_{\bar a}\supseteq t\uparrow\cap\Freebar p$, which is the first part of
(\ref{it:Ker}). Now for each $k\leq\kappa$ the second part of (\ref{it:Esa}) gives $x\in\oc
t\fc\setminus\oc s_k\fc$ such that $x\cap\Freebar p$ contains $\fKer\pi_{\bar a}$.
So $G\subseteq x$, in particular $s_k\notin G$. 
\end{proof}

\begin{theorem}[Open Mapping for $\ESFbar p$]\label{th:OMT} 
  For every integer $d$, every $y\in\ESFbar p$ and every
  $x_0\in\ESFbar{p,q}$, if $d_{\bar p}(y,x_0\cap\Freebar p)<2^{-h_{l,d}}$,
  there exists $x\in\ESFbar{p,q}$ such that $y=x\cap\Freebar p$ and
  $d_{\bar p,q}(x,x_0)<2^{-d}$.
\end{theorem}

This result says that the dual of the natural embedding $f:\Freebar
p\to\Freebar{p,q}$, is an open map $f^\uparrow:\ESFbar p\to\ESFbar{p,q}$. As shown
in \cite{gool-regg-2018}, it easily follows that for every
$\LHA$\--morphism $f:A\to B$ between finitely presented Heyting
algebras, the dual map $f^\uparrow:\ES B\to\ES A$ is an open map. 

\begin{proof}
  Let $y_0=x_0\cap\Freebar p$ and $B=B_{\bar p}(y_0,2^{-h_{l,d}}) =
  \oc\varphi_B\fc\setminus\oc\psi_B\fc$. Let $A_0=\Freebar p/y_0$,
  $B_0=\Freebar{p,q}/x_0$ and $(\bar a_0,b_0)$ be the image of $(\bar
  p,q)$ by the canonical projection $\pi_{\bar a_0,b_0}:\Freebar{p,q}\to
  B_0$. By construction $A_0$ is the Heyting algebra generated by
  $\bar a_0$ in $B_0$, and the system $\EQUATION\cS{\varphi_B}{\psi_B}{a_0}q$ has a
  solution in $B_0$, namely $b_0$. 

  Let $A=\Freebar p/y$, and $\bar a$ be the image of $\bar p$ by the
  canonical projection $\pi_{\bar a}:\Freebar p\to A$. By construction
  $d_{\bar p}(y_0,2^{-h_{l,d}})<2^{-h_{l,d}}$, that is
  $y\sim_{h_{l,d}}y_0$. Hence for every IPC formula $\theta(\bar p)$ of degree
  $\leq h_{l,d}$, $\theta(\bar p)\in y$ if and only if $\theta(\bar p)\in y_0$. But
  $\theta(\bar p)\in y$ if and only if $\theta(\bar a)=\UN$, and similarly for
  $\bar a_0$, so $\Th_{h_{l,d}}(\bar a)=\Th_{h_{l,d}}(\bar a_0)$.

  Now $A$ and $A_0$ are non-trivial by construction, so they contain
  the two-element Heyting algebra. By the amalgamation property,
  both of them embed in a common extension, which itself embeds in an
  existentially closed extension $H$. So $\EQUATION\cS{\varphi_B}{\psi_B}{a_0}q$
  has a solution in $A_0$, hence in $H$, and $\bar a\approx_{h_{l,d}}\bar
  a_0$. Now $h_{l,d}(H)=h_{l,d}$ by Theorem~\ref{th:axiom-mod-comp}, so 
  $\EQUATION\cS{\varphi_B}{\psi_B}{a}q$ has a solution $b$ in $H$. The
  equivalence (\ref{it:zero})$\Leftrightarrow$(\ref{it:Esa}) of
  Lemma~\ref{le:zero-Ker-Esa} applies: since $y=\fKer\pi_{\bar a}$,
  there is $x\in\oc\varphi_B\fc\setminus\oc\psi_B\fc$ (that is $x\in B$) such that
  $x\cap\Freebar p=y$. 
\end{proof} 

% The condition $d_{\bar p}(y,x_0\cap\Freebar p)<2^{-h_{l,d}}$ in our Open
% Mapping Theorem~\ref{th:OMT} is optimal. Indeed, if $n<h_{l,d}$, by
% definition of $h_{l,d}=h_{l,d}(\HH)$ there exists $\bar a,\bar
% a_0\in\HH^l$ and a system $\EQpar p$ of degree $\leq d$, with $\bar
% s=(s_1,\dots,s_\kappa)$, such that: $\bar a\approx_n\bar a'$, $\EQpar a_0$ has a
% solution in $\HH$ and $\EQpar a$ has no solution in $\HH$. By
% Lemma~\ref{le:zero-Ker-Esa}, for every $k\leq\kappa$ there are $y_k\in\ESFbar p$
% and $x_k\in\ESFbar{p,q}$ such that $y_k=x_k\Freebar p$ contains
% $\fKer\pi_{\bar a'}$ and $x_k\in\oc t\fc\setminus\oc s_k\fc$. Since $\HH$
% is existentially closed $\EQpar q$ has no solution in any extension of
% the Heyting algebra generated by $\bar a$ in $\HH$, hence by
% Lemma~\ref{le:zero-Ker-Esa} for every $x\oc t\fc$ which contains
% $\fKer\pi_{\bar a}$ there is at least which contains
% $\fKer\pi_{\bar a}$, 

\section{Discriminant and co-discriminant}
\label{se:disc}

Throughout this section we fix a system of degree $\leq d$
\begin{displaymath}
  \EQpar p=\Big[t(\bar s,p)=\UN\land\lland_{k\leq \kappa}s_k(\bar
  p,q)\neq\UN\Big].
\end{displaymath}
Like the usual discriminant does for polynomial equations, we want to
find IPC formulas in the coefficients of the system whose values at
any $l$\--tuple $\bar a$ in a Heyting algebra $A$ decides if $\EQpar
a$ has a solution in some extension of $A$.

We define the {\df discriminant} of $t$ as
\begin{displaymath}
  \Delta_t(\bar p)=\mmeet_{B\in\cD_t}\big(\varphi_B(\bar p)\to\psi_B(\bar p)\big)
\end{displaymath}
where $\cD_t$ is the set of balls $B$ in $\ESFbar p$ of radius
$2^{-R_{l,d}}$ such that
\begin{displaymath}
  t(\bar p,q)\meet \varphi_B(\bar p) \leq \psi_B(\bar p).
\end{displaymath}
For each $k\leq\kappa$, the {\df co\--discriminant} of $(t,s_k)$ is
\begin{displaymath}
  \nabla_{t,s_k}(\bar p)=\mmeet_{B\in\cD'_{t,s_k}}\big(\varphi_B(\bar p)\to\psi_B(\bar p)\big)
\end{displaymath}
where $\cD'_{t,s_k}$ is the set of balls $B$ in $\ESFbar p$ of radius
$2^{-R_{l,d}}$ such that 
\begin{displaymath}
  t(\bar p,q)\meet \varphi_B(\bar p) \nleq
  s_k(\bar p,q)\join \psi_B(\bar p).
\end{displaymath}

\begin{theorem}\label{th:disc}
  For every $l$\--tuple $\bar a$ in a Heyting algebra
  $A$, the system 
  \begin{displaymath}
    \EQpar p=\Big[t(\bar s,p)=\UN\land\lland_{k\leq \kappa}s_k(\bar
    p,q)\neq\UN\Big].
  \end{displaymath}
  has a solution in some extension of $A$ if and only if $\Delta_t(\bar
  a)=\UN$ and $\nabla_{t,s_k}(\bar a)\neq\UN$ for every $k\leq\kappa$. 
\end{theorem}

The proof is based on Lemma~\ref{le:zero-Ker-Esa} and the Open
Mapping Theorem. 

\begin{proof} (of Theorem~\ref{th:disc})
  By Lemma~\ref{le:zero-Ker-Esa} the system $\EQpar a$ has a solution in
  some extension of $A$ if and only if $(i)$ for every $y\in\ESFbar p$
  containing $\fKer\pi_{\bar a}$ there is $x\in\ESF{\bar p,q}$ such that
  $x\cap\ESFbar p=y$ and $t\in x$, and for each $k\leq\kappa$; $(ii)_k$ there are
  $y\in\ESFbar p$ and $x\in\ESF{\bar p,q}$ such that $x\cap\ESFbar p=y$
  contains $\fKer\pi_{\bar a}$, $t\in x$ and $s_k\notin x$. 

  Now for every ball $B$ in $\ESFbar p$ of radius $2^{-R_{l,d}}$ such
  that $\varphi_B(\bar a)\to\psi_B(\bar a)\neq\UN$, by Claim~\ref{cl:phiB-y-Ker}
  there is $y\in B$ containing $\fKer\pi_{\bar a}$, so $(i)$ implies
  that there is also $x\in\ESF{\bar p,q}$ such that $x\cap\ESFbar p=y$
  and $t\in x$. In particular $x$ and $y$ contain the same elements of
  $\Freebar p$, hence $\psi_B\in x$ and $\psi_B\notin x$ since $y\in
  B=\oc\varphi_B\fc\setminus\oc\psi_B\fc$. Altogether $t\meet\varphi_B\in x$ and $\psi_B\notin x$,
  consequently $t\meet\varphi_B\nleq\psi_B$ in $\Free{\bar p,q}$, that is $B\notin\cD_t$.
  Equivalently, $(i)$ implies that $\varphi_B(\bar a)\to\psi_B(\bar a)=\UN$ for
  every ball $B\in\cD_t$, that is $\Delta_t(\bar a)=\UN$. 

  Conversely, let us prove that $\Delta_t(\bar a)=\UN$ implies $(i)$. Pick
  any $y\in\ESFbar p$ containing $\fKer\pi_{\bar a}$. Let $B$ be the
  unique ball in $\ESFbar p$ of radius $2^{-R_{l,d}}$ containing $y$.
  Then $\varphi_B(\bar a)\to\psi_B(\bar a)\neq\UN$ by Claim~\ref{cl:phiB-y-Ker}, so
  $B\notin\cD_t$ because $\Delta_t(\bar a)=\UN$. By definition of $\cD_t$ we
  have $t\meet\varphi_B\nleq\psi_B$, so there is a point $x'\in\ESF{\bar p,q}$ such that
  $t\meet\varphi_B\in x'$ and $\psi_B\notin x'$. Let $y'=x'\cap\ESFbar p$, since $y'$ and
  $x'$ contain the same elements in $\Freebar p$ we have $\varphi_B\in y'$ and
  $\psi_B\notin y'$, that is $y'\in\oc\varphi_B\fc\setminus\oc\psi_B\fc=B$. Now, since $y\in B$ and
  $B$ has radius $2^{-R_{l,d}}$, the Open Mapping Theorem gives
  $x\in\ESF{\bar p,q}$ such that $x\sim_d x'$ and $x\cap\ESFbar p=y$. In
  particular $t\in x'$ implies that $t\in x$, that is $x\in\oc t\fc$, which
  proves $(i)$.

  Fix now some $k\leq;K$ and assume that $\nabla_{t,s}(\bar a)\neq\UN$, so there
  is a ball $B$ in $\ESFbar p$ of radius $2^{-R_{l,d}}$ such that:
  $(a)_k$ $t\meet\varphi_B\nleq s_k\join\psi_B$, and; $(b)$ $\varphi_B(\bar a)\to\psi_B(\bar a)\neq\UN$.
  Then $(a)_k$ gives some $x'\in\ESF{\bar p,q}$ which contains $t$ and
  $\varphi_B$ but neither $s_k$ nor $\psi_B$. Let $y'=x'\cap\Freebar p$, then
  $\varphi_B\in y'$ and $\psi_B\notin y'$, because $y'$ and $x'$ contain the same
  elements in $\Freebar p$. So $y'$ belongs to
  $\oc\varphi_B\fc\setminus\oc\psi_B\fc=B$. By $(b)$ and Claim~\ref{cl:phiB-y-Ker}
  there is some $y\in B$ which contains $\fKer\pi_{\bar a}$. Since
  $y'=x'\cap\Freebar p$ and $y,y'\in B$, the Open Mapping Theorem then
  gives $x\in\ESF{\bar p,q}$ such that $x\sim_n x'$ and $x\cap\Freebar p=y$.
  In particular $t\in x$ and $s_k\notin x$ because $t$, $s_k$ have degree $\leq
  d$ and $x\sim_d x'$, which proves $(ii)_k$. 

  It only remains to check that $(ii)_k$ implies that $\nabla_{t,s_k}(\bar
  a)\neq\UN$. By assumption $(ii)_k$ there are some $y\in\ESFbar p$ and
  $x\in\ESF{\bar p,q}$ such that  $t\in x$ and $s_k\notin x$, and such that
  $y=x\cap\Freebar p$ contains $\fKer\pi_{\bar a}$. Let $B$ be the unique
  ball in $\ESFbar p$ of radius $2^{-R_{l,d}}$ containing $y$. Then
  $\varphi_B\in x$ and $\psi_B\notin x$ because $x$ and $y$ contain the same
  elements of $\Freebar p$. Altogether $t\meet\varphi_B\in x$ and $s_k\join\psi_B\notin x$, so
  $t\meet\varphi_B\nleq s_k\join\psi_B$, that is $B\in\cD'_{t,s_k}$. Moreover $y$ belongs to
  $B$ and contains $\fKer\pi_{\bar a}$ so $\varphi_B(\bar a)\to\psi_B(\bar a)\neq\UN$,
  and {\it a fortiori} $\nabla_{t,s_k}(\bar a)\neq\UN$.
\end{proof}

\section{Appendix: the Finite Extension Property}
\label{se:appendix}

We collect here all the results from \cite{darn-junk-2011} needed for
the proof of the Finite Extension Property (Theorem~\ref{th:FEP-apx}
below). The classical Finite Model Property only concerns
satisfiability of formulas $\exists \bar p,\,t(\bar p)\neq\UN$ with $t(\bar p)$
an IPC\--formula. The slight improvement below is Proposition~8.1 in
\cite{darn-junk-2011}. 

\begin{proposition}\label{pr:EFMP} %[Extended Finite Model Property]
  Let $\cV$ be a variety of Heyting algebras having the Finite
  Model Property. If an existential $\LHA$\--formula is satisfied
  in some $\cV$\--algebra, then it is satisfied in a finite
  $\cV$\--algebra.
\end{proposition}

Recall that the {\df height} of a poset $(E,\leq)$ is the maximal integer
$n$ such that there exists $a_0<a_1<\cdots<a_n$ in $E$. If no such maximum
exists we say that $E$ has infinite length. The {\df dimension} of a
lattice $L$, denoted $\dim L$, is the height of its prime filter
spectrum (or equivalently of its prime ideal spectrum). This is also
the classical Krull dimension of $\Spec^\uparrow L$ and of $\Spec_\downarrow L$. 

We introduced in \cite{darn-junk-2011} a notion of dimension and
co-dimension for the elements of a lattice, which turned out to have
better properties in co-Heyting algebras. Given an element $a$ of a
Heyting algebra $A$, we define the {\df dual co-dimension} of $a$ in
$A$ to be the co-dimension of $a$ in $A^\text{op}$, the co-Heyting
algebra obtained by reversing the order of $A$. Equivalently, by
Theorem~3.8 in \cite{darn-junk-2011}, it is the greatest integer $d$
such that 
\begin{displaymath}
  \exists x_0,\dots,x_d\in A,\ a\geq x_d\gg\cdots\gg x_0.
\end{displaymath}
\begin{remark}\label{re:codim-sub}
  An immediate consequence of this characterisation is that if $B$
  is a Heyting algebra such that $a\in B\subseteq A$, the dual co-dimension of
  $a$ in $A$ is greater than or equal to its dual co-dimension in
  $B$.
\end{remark}

We let $dA$ denote\footnote{This is slightly different from our
  notation in \cite{darn-junk-2011}, where we let $dA$ be the set of
elements of co-dimension $\geq d$.} the set of elements in $A$ with
dual co-dimension $>d$. Let us recall some of the properties of the
dual co-dimension proved, after reversing the order, in
\cite{darn-junk-2011}. 

{\sl
  \begin{description}
    \item
      [\CD 1] A Heyting algebra $A$ has dimension $\leq d$ if and only
      if every $a\in A\setminus\{\UN\}$ has dual co-dimension $\leq d$ in $A$, that
      is $(d+1)A=\{\UN\}$ (Remark~2.1).
    \item
      [\CD 2] There is an IPC formula $\varepsilon_{l,d}(\bar p)$ such that
      for every Heyting algebra $A$ generated by an $l$\--tuple
      $\bar a$, $dA$ is the filter generated in $A$ by $\varepsilon_{l,d}(\bar
      a)$ (Proposition~8.2).
    \item
      [\CD 3] For every finitely generated Heyting algebra $A$,
      $A/dA$ is finite (Corollary~5.5 and Corollary 7.5).
  \end{description}
}

\begin{remark}\label{re:dim-sub}
  From \CD 1 and Remark~\ref{re:codim-sub} it follows that if a
  Heyting algebra $A$ has finite dimension $d$, then every Heyting
  algebra contained in $A$ has dimension $\leq d$. This also follows
  from \cite{hoso-1967} and \cite{ono-1970}, who proved {\it mutatis
  mutandis}, that the class of Heyting algebras of dimension $\leq d$
  is a variety (a direct proof of this is given in
  \cite{darn-junk-2011}, Proposition~3.12). 
\end{remark}

\begin{lemma}\label{le:FEP}
  Given an integer $d\geq0$ there is an
  integer $n_d\geq0$ such that for every Heyting algebra $H$ of dimension
  $\leq d$, every $\bar a\in H^l$ and every $l$\--tuple $\bar a'$ in a Heyting
  algebra $H'$, if $\Th_n(\bar
  a)=\Th_n(\bar a')$ then $\fKer\pi_{\bar a}=\fKer\pi_{\bar a'}$. 
\end{lemma}

\begin{proof}
  By Remark~\ref{re:dim-sub}, for every Heyting algebra $H$ of dimension
  $\leq d$, and every $\bar a\in H^l$, the Heyting algebra $A$ generated by
  $\bar a$ in $H$ has dimension $\leq d$. Then
  $(d+1)A=\UN$ by \CD 1, that is $\varepsilon_{l,d}(\bar a)=\UN$ by \CD 2.
  In particular, $\varepsilon_{l,d}(\bar p)\in\fKer\pi_{\bar a}$ hence
  $(d+1)\Freebar p$ is contained in $\fKer\pi_{\bar a}$. By \CD 3,
  $\Freebar p/(d+1)\Freebar p$ is finite. Hence there exists only
  finitely many prime filters $y$ in $\Freebar p$ which contain
  $(d+1)\Freebar p$, and since the filter $(d+1)\Freebar p$ is
  principal, every such $y$ is principal (here we use
  Fact~\ref{fa:ideal-quo}). 
  
  So let $\xi_1(\bar p),\dots,\xi_\alpha(\bar p)$ be a list of IPC formulas such
  that every $y\in\ESFbar p$ containing $(d+1)\Freebar p$ is generated
  by $\xi_i(\bar p)$ for some $i\leq \alpha$. Let $n_d$ be the maximum of the
  degrees of the $\xi_i(\bar p)$'s and of $\varepsilon_{l,d}(\bar p)$ and assume
  that $\bar a'\approx_n\bar a$. Then $\varepsilon_{l,d}(\bar a')=\UN$, that is
  $\varepsilon_{l,d}(\bar p)\in\fKer\pi_{\bar a'}$ so $(d+1)\Freebar p$ is contained
  in $\fKer\pi_{\bar a'}$. Thus, for every $y\in\ESFbar p$, $y$ contains
  $\fKer\pi_{\bar a}$ (resp. $\fKer\pi_{\bar a'}$) if and only if $y$ is
  generated by $\xi_i(\bar p)$ for some $i\leq \alpha$ such that $\xi_i(\bar
  a)=\UN$ (resp. $\xi_i(\bar a')=\UN$). Since $\Th_n(\bar a)=\Th_n(\bar
  a')$ by assumption, $\xi_i(\bar a)=\UN$ if and only if $\xi_i(\bar
  a')=\UN$. This ensures that $\fKer\pi_{\bar a}$ and $\fKer\pi_{\bar a'}$
  are contained in the same prime filters of $\Freebar p$. By
  Fact~\ref{fa:radical} it follows that $\fKer\pi_{\bar a}=\fKer\pi_{\bar
  a'}$.
\end{proof}

\begin{remark}\label{re:dim-N-deg-2N+1}
  In \cite{bell-1986} an explicit formula $\xi(\bar p)$, with degree
  $\leq 2d+1$, is given for every join-irreducible element of $\Freebar
  p$ of foundation rank $\leq d$ (see also Theorem~3.3 in
  \cite{darn-junk-2010b}). A prime filter $y$ of $\Freebar p$
  contains $(d+1)\Freebar p$ if and only if is generated by such a
  join-irreducible element $\xi(\bar p)$, so we can take $n_d=2d+1$ in
  Lemma~\ref{le:FEP}.
\end{remark}

\begin{theorem}[Finite Extension Property]\label{th:FEP-apx}
  Let $\cV$ be a variety of Heyting algebras having the Finite Model
  Property. If an existential $\LHA$\--formula with parameters in a
  finite $\cV$\--algebra $A$ is satisfied in a $\cV$\--algebra
  containing $A$, then it is satisfied in some finite $\cV$\--algebra
  containing $A$. 
\end{theorem}

\begin{proof}
  It suffices to prove the result for a formula $\exists\bar q\,\EQparsol a
  {\bar q}$ with $\bar a\in A^l$ generating $A$. Let $d$ be the degree
  of $\EQparsol p {\bar q}$. Let $n\in\NN$ be given by Lemma~\ref{le:FEP}
  applied to the integer $\dim A$ (since $A$ is finite, $A$ is also
  finite dimensional). Up to intuitionist equivalence there are
  finitely many IPC formulas $\varphi(\bar p)$ of degree $\leq n$. Let $u$ be
  the conjunction of those which belong to $\Th_n(\bar a)$, and $\bar
  v=(v_1,\dots,v_\alpha)$ be an enumeration of the others. Now consider the
  system $\EQUATION S {t'}{s'} p q$ where $t'=t\meet u$ and $\bar
  s'=(s_1,\dots,s_\kappa,v_1,\dots,v_\alpha)$. 

  By assumption $\EQparsol a {\bar q}$ has a solution $\bar b$ in some
  extension of $A$. Then $(\bar a,\bar b)$ is a solution of $\EQUATION
  S {t'}{s'} p {\bar q}$. By Proposition~\ref{pr:EFMP}, $\EQUATION S
  {t'}{s'} p {\bar q}$ then has a solution $(\bar a',\bar b')$ in a
  finite Heyting algebra $H'$ in $\cV$. In particular $b'$ is a
  solution $\EQparsol{a'}{\bar q}$. Moreover, $u(\bar a')=\UN$ and
  $v_i(\bar a')\neq\UN$ for every $i\leq \alpha$, hence $\Th_n(\bar
  a')=\Th_n(\bar a)$. By assumption on $n$, this implies that
  $\fKer\pi_{\bar a}=\fKer\pi_{\bar a'}$. So there is an isomorphism from
  $A$ to the Heyting algebra $A'$ generated by $\bar a'$ which maps
  $\bar a$ onto $\bar a'$. By embedding $A$ into $H'$ via this
  isomorphism, we get a finite $\cV$\--algebra containing $A$ in which
  $\EQparsol a {\bar q}$ has a solution.
\end{proof}

\paragraph{Acknowledgement} I would like to thank Sam van Gool and
Luca Reggio for helpful discussions during the TOLO VI conference a
Tbilissi, Georgia, and to express my gratitude to the organizers of
this meeting hosting these discussions.

% \bibliographystyle{alpha}
% \bibliography{biblio}

\end{document}